\documentclass[12pt]{amsart}
\usepackage[a4paper,top=1in,bottom=1in,left=1in,right=1in,marginparwidth=1.75cm]{geometry}

\usepackage{amsmath,amssymb,mathtools}
\usepackage{enumitem,makecell,caption,longtable}
\setlist[enumerate,1]{font=\textnormal} 
\usepackage[colorlinks=true, allcolors=blue]{hyperref}

\newcommand{\mbf}[1]{\mathbf{#1}}
\newcommand{\mr}[1]{\mathrm{#1}}

\newcommand{\gen}[1]{\langle {#1} \rangle}

\newcommand{\Z}{\mathbb{Z}}
\newcommand{\C}{\mathbb{C}}
\newcommand{\Zen}[1]{\mathbf{Z}(#1)}

\newcommand{\Cen}[2]{\mathbf{C}_{#1}(#2)}
\newcommand{\Norm}[2]{\mathbf{N}_{#1}(#2)}
\newcommand{\nrm}{\trianglelefteq}
\newcommand{\Irr}[1]{\mathrm{Irr}(#1)}
\newcommand{\wh}{\widehat}
\newcommand{\q}[1]{\overline{#1}}
\newtheorem{thmx}{Theorem}

\theoremstyle{definition}
\newtheorem{theorem}{Theorem}[section]

\newtheorem{proposition}[theorem]{Proposition}
\newtheorem{example}[theorem]{Example}
\newtheorem{remark}[theorem]{Remark}
\newtheorem{conjecture}[theorem]{Conjecture}
\newtheorem{corollary}[theorem]{Corollary}

\title{Refining invariants of finite groups with class functions}

\author{Christopher A. Schroeder}
\address{HUN-REN Alfr\'ed R\'enyi Institute of Mathematics, Re\'altanoda utca 13-15, H-1053 Budapest, Hungary}
\email{schroeder.chris@renyi.hu}
\date{\today}

\begin{document}

\begin{abstract}
Many numerical invariants of a finite group arise as the multiplicity of the trivial character in a naturally associated class function, and the remaining multiplicities carry finer structural information. For the number of conjugacy classes of elements whose order involves only a chosen set $\pi$ of primes, the
natural such class function is a generalized character interpolating between the
classical conjugating character and the regular character.
Robinson asked whether this generalized character is always a character. We prove that it is a character for $\pi=p'$ and for $\pi$-separable groups, cases which Robinson obtained independently, and we show that the irreducible Weil character of $\mathrm{PSU}_3(q)$ of degree $q^2-q$, where $(3,q+1)=1$, occurs in it with nonnegative multiplicity for every $\pi$.
Furthermore, we characterize when these class functions are compatible with passing to a subgroup and thereby sharpen a theorem of Sangroniz. The analogous class function for real elements is shown to always give a character. Finally, we give a dual construction that, together with the recently proved McKay Conjecture, yields a best-possible criterion for a finite group to have an abelian Sylow $p$-subgroup.
\end{abstract}

\maketitle

\section{Introduction}
There is a long history of using invariants to unravel the structure of mathematical objects. By a group invariant we mean a function assigning to each finite group $G$ a complex number $I_G$ such that $I_G = I_H$ whenever $G \cong H$. As nonisomorphic groups could have the same value of an invariant, one would often like to refine, or strengthen, an invariant to obtain more detailed information about group structure. We take up this problem here in the following way: Given an invariant $I_G$, we seek a complex-valued class function $\mathbf{I}_G:G\rightarrow \C$ such that, in a chosen orthonormal basis $\{e_1,\dots,e_n\}$ of the space of class functions, the class function $\mbf{I}_G$ has our invariant $I_G$ as its inner product $[\mbf{I}_G,e_1]$ with its fundamental component~$e_1$. Such a class function always exists, and the construction need not be unique; indeed, $\mbf{I}_G = I_G \, e_1$ is a trivial way to construct such a class function. However, even trivial constructions can be interesting, as we will see in Example~\ref{ex:sang}. In general, our goal is to find a natural, interesting or useful class function~$\mbf{I}_G$ satisfying $I_G = [\mbf{I}_G,e_1]$.

For example, let $k(G)$ be the number of conjugacy classes of $G$, and let $\Pi_G$ be the so-called conjugating character of $G$, defined by $\Pi_G(g)=|\Cen{G}{g}|$ for all $g \in G$; then $k(G)=[\Pi_G,1_G]$, where $1_G$ is the trivial character of $G$. A so-called ``$\pi$-local'' version of this construction will be the main concern of this paper. Let $\pi$ be a collection of primes, and let $k_\pi(G)$ be the number of $G$-conjugacy classes of $\pi$-elements. Then the class function $\Pi_{\pi,G}$ defined by $\Pi_{\pi,G}(g)=|\Cen{G}{g}|$ if $g \in G$ is a $\pi$-element and zero otherwise again satisfies $k_\pi(G)=[\Pi_{\pi,G},1_G]$. We also define the class function $\Delta_{\pi,G} = \Pi_{\pi,G}/|G|_\pi$ yielding the scaled invariants $d_\pi(G)=k_\pi(G)/|G|_\pi$. Section~2 contains a more detailed discussion of these constructions and others like them from the literature.

Our refining construction has an analogy in Fourier analysis and music: If our group~$G$ were a tuned drum, then $\mbf{I}_G$ captures its frequency spectrum when struck, where~$I_G=[\mbf{I}_G,e_1]$ is the weight of the fundamental frequency and the multiplicities $I_G^{e_i}:=[\mbf{I}_G,e_i]$ ($i>1$) are the weights of the ``higher harmonics''. Seen in this light, one might hope that $\mathbf{I}_G$ can discriminate between groups that have the same value of $I_G$ in the same way that some musical instruments, such as a trumpet and a cornet, have the same fundamental frequency but can be distinguished by their distributions of higher frequencies, their so-called timbre.

Given such a class function $\mathbf{I}_G$, some natural questions immediately arise:
\begin{enumerate}[nolistsep,label=(\arabic*)]
    \item Is $\mathbf{I}_G$ a (generalized) character? Does there exist some $c \in \C$ such that $c \mathbf{I}_G$ is a (generalized) character?
    \item What can we say about the multiplicities $I_G^{e_i}$ ($i>1$)? For example, when are they nonzero? If they are integers, do they encode any interesting information?
    \item What group structure does $\mathbf{I}_G$ capture? For example, what is the relationship between a group $G$ and a subgroup $H \le G$ if $\mbf{I}_G=(\mbf{I}_H)^G$?
\end{enumerate}
Note that if we choose $e_1$ to be the trivial character $1_G$, then the condition that \mbox{$\mbf{I}_G=(\mbf{I}_H)^G$} in Question~3 is a true refinement of $I_G = I_H$ since we have by Frobenius Reciprocity~that
\begin{align*}
    I_G = [\mbf{I}_G,1_G] = [(\mbf{I}_H)^G,1_G] = [\mbf{I}_H,1_H] = I_H.
\end{align*}

Our main results answer Questions~(1)--(3) for $\Pi_{\pi,G}$. Regarding Question~(1), we prove the following. For a fixed prime $p$, recall that $\pi=p'$ is the set of all primes except~$p$. 

\begin{thmx}\label{thmA}
Let $G$ be a finite group, and let $\pi$ be a collection of primes. The class function $\Pi_{\pi,G}$ is a generalized character of $G$. Furthermore, if
\begin{enumerate}[nolistsep,label=(\roman*)]
    \item $\pi = p'$ for some prime $p$ and $G$ is an arbitrary finite group or
    \item $\pi$ is arbitrary and $G$ is a $\pi$-separable finite group,
\end{enumerate}
then $\Pi_{\pi,G}$ is a character of $G$. 
\end{thmx}

\noindent
The two cases of Theorem~\ref{thmA} were obtained independently by Robinson~\cite{Ro25}.
Our proof in the case $\pi = p'$ is different from his; our argument in the $\pi$-separable case is very similar, so we only sketch it. As for the rescaling asked about in Question~(1), we characterize exactly when $\Delta_{\pi,G}$ is a
(generalized) character: this happens if and only if $d_\pi(G) = 1$, if and only if $G$ has a normal $\pi$-complement and an abelian Hall $\pi$-subgroup (Theorem~\ref{prop2}). 
Whether $\Pi_{\pi,G}$ is always a character is a question of Robinson, which we record as Conjecture~\ref{c:Picharacter}. A counterexample would consist of a finite group $G$, a collection $\pi$ of primes and a character $\chi \in \mathrm{Irr}(G)$ with $[\Pi_{\pi,G},\chi] < 0$. By~\cite{heide13, heide06}, the irreducible characters of nonabelian simple groups occurring with multiplicity zero in the conjugating character are precisely the Weil characters of degree $(q^n-q)/(q+1)$ of $\mathrm{PSU}_n(q)$ with $n \ge 3$ coprime to $2(q+1)$; such pairs $(G,\chi)$ are a natural place to look for a counterexample. The following theorem rules out the pairs with $n=3$.

\begin{thmx}\label{thmB}
If $G \cong \mr{PSU}_3(q)$ with $(3,q+1)=1$ and $\chi$ is the Weil character of degree~\mbox{$q^2-q$}, then $[\Pi_{\pi,G},\chi]\ge 0$ for all collections $\pi$ of primes.
\end{thmx}

\noindent
The analogous class function built from the real elements of $G$, rather than from its $\pi$-elements, is always a character. Recall that a group element is real if it is conjugate to its inverse.

\begin{thmx}\label{thmC}
Let $G$ be a finite group, and let $R \subseteq G$ be the collection of real elements of~$G$. Then the class function equal to $|\Cen{G}{g}|$ for $g \in R$ and $0$ otherwise is a character of~$G$.
\end{thmx}

Question~(2) we are able to answer only in part. We show that $\Pi_{\pi,G}$ has a single irreducible constituent only in degenerate cases (Theorem~\ref{prop5}), and, as a byproduct of the proof of Theorem~\ref{thmA}, that the row sums in the $p$-Brauer character table are nonnegative integers (Remark~\ref{r:rowsums}), sharpening a result of Chillag~\cite[Proposition~2.2(b)]{chillag98}. 

Regarding Question~(3), we show that $\Pi_{\pi,G}=(\Pi_{\pi,H})^G$ if and only
if there is a bijection of a certain form between the $\pi$-conjugacy classes of $G$ and $H$ (Theorem~\ref{t:Xclasses}), which we prove in the greater generality of Example~\ref{ex:X}. Our main result here is the representation-theoretic counterpart, which sharpens a theorem of Sangroniz \cite[Theorem 3.3]{S08} that characterizes when $d_{p'}(G)=d_{p'}(H)$ for a $p$-solvable group $G$.

\begin{thmx}\label{thmD}
Let $p$ be a prime, and let $H$ be a subgroup of the finite group $G$. If restriction is a bijection between the irreducible Brauer characters of $G$ and $H$ and $[G:H]_{p'} = 1$, then $\Delta_{p',G}=(\Delta_{p',H})^G$. If $G$ is $p$-solvable, then the converse holds.
\end{thmx}

Finally, in Section 5 we study the invariant $|\mr{Irr}_{p'}(G)|/|G|$, which arises from a dual construction in which an invariant is encoded in the multiplicities of the irreducible constituents of a class function rather than in its values. We obtain the following criterion, which we show to be best possible.

\begin{thmx}\label{thmE}
Let $G$ be a finite group, $p$ a prime and $P$ a Sylow $p$-subgroup of $G$. If $|\mr{Irr}_{p'}(G)|/|G| > 1/p$, then $P$ is abelian. If, in addition, $p$ is the smallest prime dividing~$|G|$, then $G$ has a normal $p$-complement.
\end{thmx}

The layout of the paper is as follows. In Section~\ref{s:setup}, we define two natural bases for the space of complex class functions on a group and discuss some previous results from the literature that can be recast in the class function language. In Sections~\ref{s:props}--\ref{s:Q3}, we prove our main results regarding $\Pi_{\pi,G}$ and $\Delta_{\pi,G}$, including Theorems~A--D. Finally, in Section~\ref{s:dual}, we prove Theorem~E.

{\bf Notation.} Our notation is standard. All groups in this paper are finite. Let $G$ be a group. For each element $x \in G$, we denote by $x^G=\{x^g \mid g \in G\}$ its $G$-conjugacy class. The collection of conjugacy classes is denoted by $\mr{Cl}(G)$, and their number is $k(G)$. We let $x_i \in G$ for $i=1,\dots,k(G)$ be representatives for the $G$-conjugacy classes. The collection of irreducible complex characters of $G$ is denoted $\Irr{G}$. It is known that $k(G)=|\Irr{G}|$. Let $\chi_1$ be the trivial character, which we often denote as $1_G$. The irreducible characters $\chi_i \in \Irr{G}$ are orthonormal with respect to the inner product $[\chi_i, \chi_j] = (1/|G|) \sum_{g \in G} \chi_i(g) \q{\chi_j(g)}$. Let $\mr{cf}(G)$ be the vector space of complex-valued class functions on $G$. Now let $p$ be a prime. The collection of $G$-conjugacy classes consisting of elements of order coprime to~$p$, also called $p'$-elements or $p$-regular elements, is denoted by $\mr{Cl}_{p'}(G)$. The number of $p$-regular conjugacy classes of $G$ is $k_{p'}(G)=|\mr{Cl}_{p'}(G)|$. Let $\mr{IBr}_p(G)$ be the irreducible $p$-Brauer characters of~$G$. Now let $\pi$ be a collection of primes. Let $G_\pi \subseteq G$ denote the set of $\pi$-elements of~$G$, let $\mr{Cl}_\pi(G)$ denote the collection of $G$-conjugacy classes of $\pi$-elements, let $|G|_\pi$ denote the $\pi$-part of the integer $|G|$, and let $k_{\pi}(G)$ be the number of $\pi$-conjugacy classes of $G$.

\section{Examples}\label{s:setup}

In this section, we give some examples from the literature that can be formulated using the class function language described in the introduction. In our first collection of examples, we choose a basis for the space $\mr{cf}(G)$ of complex class functions on a finite group $G$ consisting of the irreducible ordinary characters $\Irr{G}$. 

\begin{example}(Element orders.)\label{ex:moreto} 
The invariant $I_G = \sum_{g \in G} |g|$ summing the element orders has been much studied from a group theoretical point of view; see~\cite{amiri09, amiri26, herzog18, herzog22, La23,La20,T14,T20} and the references therein. Moret{\'o} studied the class function defined by $\mbf{I}_G(g) = |G| |g|$ for all $g \in G$ in \cite{Mor23}, noting that $I_G=[\mbf{I}_G,1_G]$. Moret{\'o} proved, among other things, that~$\mbf{I}_G$ is a generalized character.

The closely related invariant $I_G = (1/|G|) \sum_{g \in G} |g|$ calculating the average element order has also been much studied; see~\cite{herzog22,AvOrd1, AvOrd2, La23ii}. The class function defined by $\mbf{I}_G(g) = |g|$ satisfies $I_G = [\mbf{I}_G,1_G]$. Moret{\'o} showed that the smallest positive integer $m$ such that $m\mbf{I}_G$ is a generalized character is $m=|G|$.
\end{example}

\begin{example}(Conjugacy classes.)\label{ex:conj}
As discussed in the introduction, the invariant~$I_G=k(G)$ and the class function~$\mbf{I}_G=\Pi_G$ satisfy our refining construction $I_G=[\mbf{I}_G,1_G]$. This class function is the character of the so-called conjugating representation, which is the linear representation obtained by letting $G$ act by conjugation on its group algebra~$\C G$. The study of the constituents of the conjugating representation has a long history; see~\cite{heide13,heide06} and the references therein.

The closely related invariant $I_G = k(G)/|G|$ then arises via the class function $\mbf{I}_G(g)=|\Cen{G}{g}|/|G|$. It has been pointed out by many authors that the invariant $d(G) = k(G)/|G|$ is equal to the probability that two randomly chosen elements of $G$ commute; see~\cite{ET68, G73, J69, J77}. Many studies have shown that this invariant, often called the commuting probability or commutativity degree, encodes important structural information about the group; see~\cite{E15,G06,G73,L95} and the references therein. For example, Gustafson proved in~\cite{G73} that if $d(G)>5/8$, then $G$ is abelian.
\end{example}

\begin{example}(Permutation characters.)\label{ex:gen1}
The last example can be put in a more general context. Note that if $G$ acts on a finite set $\Omega$ and $\mbf{I}_G$ is the permutation character, then $\mbf{I}_G(g) = |\mr{Fix}_\Omega(g)|$, and so by the Orbit--Counting Lemma,
\begin{align*}
    I_G = [\mbf{I}_G,1_G] = \frac{1}{|G|}\sum_{g \in G} |\mr{Fix}_\Omega(g)| = \text{\#\{$G$-orbits on $\Omega$\}}.
\end{align*}
Example~\ref{ex:conj} arises from this construction with $G$ acting on $\Omega=G$ by conjugation. 

Now let $A$ be any finitely-generated group, and let $G$ act on $\Omega=\mr{Hom}(A,G)$ by post-composition with conjugation, i.e. \mbox{$g \cdot \phi = g \phi(-) g^{-1}$}. Then $|\mr{Fix}_{\Omega}(g)|=|\mr{Hom}(A,\Cen{G}{g})|$, and so the class function defined by $\mbf{I}_G(g) = |\mr{Hom}(A,\Cen{G}{g})|$ for all $g \in G$ has the associated invariant
\begin{align*}
    I_G = [\mbf{I}_G,1_G] = \# \{\text{$G$-orbits on $\mr{Hom}(A,G)$}\}.
\end{align*}
This gives rise to more examples:
\begin{enumerate}[nolistsep,label=(\arabic*)]
    \item If $A=\Z$, then $\mr{Hom}(A,G) \cong G$, $\mr{Hom}(A,\Cen{G}{g}) \cong \Cen{G}{g}$, and we recover  $I_G=k(G)$ and $\mbf{I}_G(g)=|\Cen{G}{g}|$ for all $g \in G$ from Example~\ref{ex:conj}.
    \item If $A=F_n$ (the free group on $n$ generators), then $\mr{Hom}(A,G) \cong G^n$, $\mbf{I}_G(g)=|\Cen{G}{g}|^n$ and $I_G$ is the number of conjugacy classes of $n$-tuples under simultaneous conjugation.
    \item If $A=\Z^n$, then $\mr{Hom}(A,G)$ is the set of pairwise-commuting $n$-tuples in $G$, and $G$ acts on these tuples by simultaneous conjugation. So $I_G$ is the number of conjugacy classes of $n$-tuples with pairwise-commuting elements. 
\end{enumerate}
The invariant in (3) was considered by Levit and Shwartz in their investigations of the so-called ``higher commutativity''~\cite{LevitShwartz2026a, LevitShwartz2026b}. Commuting $n$-tuples were also considered by Erd\H{o}s and Straus~\cite{ES76}.
\end{example}

\begin{example}(Adams operations.) \label{ex:frob-schur}
Given an irreducible character $\chi \in \Irr{G}$ of a finite group $G$, the class function $\mbf{I}_{\chi,G}=\chi^{(2)}$ defined by $\chi^{(2)}(g) = \chi(g^2)$ for all $g \in G$ gives rise to the so-called Frobenius--Schur indicator of $\chi$, $I_{\chi,G} =[\mbf{I}_{\chi,G},1_G] = (1/|G|) \sum_{g\in G} \chi(g^2)$. See \cite[Chapter~4]{I94} for more information about this classical topic.

Recently, Boltje, Kleshchev, Navarro and Tiep proved a reduction of the Feit Conjecture to a local problem on finite simple groups~\cite{B25}. The authors also formulated a stronger form of the conjecture \cite[Conjecture~E]{B25} which is connected to an integer-valued invariant $S(G,\chi,n)$ studied by Boltje and Navarro in \cite{B25ii}, where $G$ is a finite group, $\chi \in \Irr{G}$ and $n$ is a positive integer. One way to construct this invariant is as the multiplicity of the trivial character in a generalized character constructed from a particular integral linear combination of Adams operations $\Psi^k(\chi)$ of $\chi$, where $\Psi^k(\chi)(g)=\chi(g^k)$ for all $g \in G$~\cite[Theorem~B]{B25ii}. Note that $\Psi^2(\chi)=\chi^{(2)}$ from the last paragraph.
\end{example}

\begin{example}(An example of Frobenius.) 
Frobenius considered the following example, which is described in~\cite[V.19.13--V.19.14]{Hu67}. Let $K$ be a conjugacy class of a finite group~$G$, let $n$ be a divisor of~$|G|$, and let $t$ be the number of elements $g \in G$ such that $g^n \in K$. The rational number $I_G=t/n$ is an invariant of the group $G$. Now, if we define the class function $\mbf{I}_G$ on~$G$ by $\mbf{I}_G(g) = |G|/n$ if $g^n \in K$ and $0$ otherwise, then we have
\begin{align*}
    [\mbf{I}_G,1_G] = \frac{1}{|G|} \sum_{g \in G} \mbf{I}_G(g) = \frac{t}{n} = I_G.
\end{align*}
In fact, the class function perspective gives interesting information about the invariant~$I_G$ in this case: Our class function $\mbf{I}_G$ can be expressed as a linear combination of irreducible characters with coefficients in $\C$. However, it is shown in \cite[Theorem~V.19.13]{Hu67} that the coefficients are actually algebraic integers lying in $\Z[\epsilon]$, where $\epsilon$ is a primitive $|G|$-th root of unity. In particular, the coefficient $[\mbf{I}_G,1_G]=I_G$ is an algebraic integer. So, $I_G$ is a rational number and an algebraic integer, which implies that $I_G$ is an integer. This is not at all obvious from the definition.
\end{example}

As mentioned in the introduction, the next example will be one of the main concerns of this paper. We will consider it in detail in Sections~\ref{s:props} and~\ref{s:Q3}.

\begin{example}(Localizing the conjugating character.)\label{ex:pi}
There has been much interest in a so-called ``$\pi$-local'' version of the invariant $d(G)$ from Example~\ref{ex:conj}; the invariant $d_{\pi}(G) = k_{\pi}(G)/|G|_{\pi}$ seems to capture much information about the $\pi$-structure of $G$~\cite{M14,M23,TV20}. In particular, if $\pi=p'$ for some prime $p$, then $d_{p'}(G)$ captures information about the modular representations of $G$ over a field of characteristic $p$~\cite{S24}. Sangroniz characterized when $d_{p'}(G)=d_{p'}(H)$ representation-theoretically for $H$ a subgroup of a $p$-solvable group~$G$ in \cite[Theorem~3.3]{S08}. As mentioned in the introduction, the class function $\Pi_{\pi,G}$ defined by $\Pi_{\pi,G}(g)=|\Cen{G}{g}|$ if $g \in G$ is a $\pi$-element and zero otherwise and its scaled version $\Delta_{\pi,G} = \Pi_{\pi,G}/ |G|_\pi$ satisfy $[\Pi_{\pi,G},1_G]=k_{\pi}(G)$ and $[\Delta_{\pi,G},1_G] = d_\pi(G)$.

We also mention Robinson's closely related generalized character defined by $\Psi_{1,p,G}(g)=|\Cen{G}{g}_p|$, the number of $p$-elements in $\Cen{G}{g}$, if $g \in G$ is $p$-regular and $0$ otherwise~\cite{Ro25}. Robinson also defines a $\pi$-variant in \cite[Section~8]{Ro25}.
\end{example}

\begin{example}(Local class functions.)\label{ex:X}
The last example can also be put in a more general context. If $X \subseteq G$ is any subset closed under conjugation by $G$, then we may define
\begin{align*}
    \Pi_{X,G}(g) = \begin{cases}
        |\mathbf{C}_G(g)| & \text{if $g \in X$} \\
        0 & \text{otherwise}.
    \end{cases}
\end{align*}
Then $[\Pi_{X,G},1_G]$ is simply the number of conjugacy classes in $X$. If $X=G_\pi$, then we recover $\Pi_{\pi,G}$ from Example~\ref{ex:pi}. Recall that an element $g \in G$ is real if it is conjugate to~$g^{-1}$, and rational if it is conjugate to every generator of $\gen{g}$. Taking $X$ to be the set of real elements of $G$, or the set of rational elements, then yields the number of real classes, or of rational classes. We will consider $\Pi_{X,G}$ in more detail in Section~\ref{s:Q3}.
\end{example}

To conclude this section, we consider a basis for the space of class functions consisting of characteristic functions, relaxing the condition that the basis vectors be normalized to~$1$: The collection of class functions $\wh{x_i}$ for $i=1,\dots,k(G)$ defined by $\wh{x_i}(g)=1$ if $g \in x_i^G$ and $0$ otherwise is clearly a basis of $\mr{cf}(G)$. It is an orthogonal but not orthonormal basis with respect to the usual inner product, however, since $[\wh{x_i}, \wh{x_j}] =  \delta_{ij} |x_i^G|/|G|$. Let $\wh{1_G}$ denote the characteristic function of the $G$-conjugacy class consisting of the identity. Note that $[\wh{1_G}, \wh{1_G}]=1/|G|$. We now reformulate a result of Sangroniz in the language of class functions using this basis.

\begin{example}($p$-regular classes.)\label{ex:sang}
Consider the group invariant $I_G =k_{p'}(G)/|G|$. Sangroniz characterized representation-theoretically when $I_G=I_H$ for $H \le G$ in \cite[Theorem~2.1]{S08}. Now define the class function $\mbf{I}_G = k_{p'}(G) \; \wh{1_G}$ for a finite group $G$, so that $I_G=[\mbf{I}_G,\wh{1_G}]$. We can recast Sangroniz's Theorem in the following form using class functions:

\begin{theorem}{\cite[Theorem~2.1]{S08}}\label{p:Sang2}
Let $H \le G$, and let $p$ be a prime. The following are equivalent:
\begin{enumerate}[nolistsep,label=(\roman*)]
    \item $\mbf{I}_G = (\mbf{I}_H)^G$, where $\mbf{I}_G=k_{p'}(G) \wh{1_G}$.
    \item Restriction is a surjection $\mr{IBr}_p(G) \rightarrow \mr{IBr}_p(H)$ and $[G:H]$ is not divisible by $p$.
\end{enumerate}
\end{theorem}

\noindent\emph{Proof.}
(i) $\Rightarrow$ (ii): If $\mbf{I}_G = (\mbf{I}_H)^G$, then $k_{p'}(G)=\mbf{I}_G(1)=(\mbf{I}_H)^G(1)=[G:H] k_{p'}(H)$; that is, $I_G = I_H$. By \cite[Theorem~2.1]{S08}, this means that $[G:H]$ is not divisible by $p$ and every irreducible $p$-Brauer character of $G$ restricts irreducibly to $H$, so restriction defines a map $\mr{IBr}_p(G) \rightarrow \mr{IBr}_p(H)$. Since every irreducible Brauer character of $H$ lies under an irreducible Brauer character of $G$ by \cite[Corollary 8.3]{N98}, we have the desired surjection.

(ii) $\Rightarrow$ (i): Since all Brauer characters restrict irreducibly and $p \nmid [G:H]$, we have $\mbf{I}_G(1) = k_{p'}(G) = [G:H]k_{p'}(H) = (\mbf{I}_H)^G(1)$ by \cite[Theorem~2.1]{S08}. This completely characterizes the equality $\mbf{I}_G = (\mbf{I}_H)^G$ since both sides vanish off the identity. \qed
\end{example}

\section{Localizing the conjugating character}\label{s:props}

In this section, we consider Questions (1) and (2) of the introduction for the ``$\pi$-local'' class functions $\Pi_{\pi,G}$ and $\Delta_{\pi,G}$ from Example~\ref{ex:pi}. We consider Question~(3) in the next section. Note that $\Pi_{\pi,G}$ interpolates between the character $\Pi_G$ of the conjugating representation and the character $\rho_G$ of the regular representation: If $G$ is a $\pi$-group then $\Pi_{\pi,G}=\Pi_G$, while if $G$ is a $\pi'$-group then $\Pi_{\pi,G}=\rho_G$.

We now turn to Question~(1) of the introduction. Note that the rational-valued class function $\Delta_{\pi,G}$ is rarely a generalized character since $[\Delta_{\pi,G},1]=d_{\pi}(G)$ and \mbox{$0 < d_\pi(G) \le 1$} by \cite[Lemma~3.5]{M21}. Our next theorem characterizes exactly when $\Delta_{\pi,G}$ is a (generalized) character.

\begin{theorem}\label{prop2} 
Let $G$ be a finite group, and let $\pi$ be a collection of primes. Then the following are equivalent:
\begin{enumerate}[nolistsep,label=(\roman*)]
    \item $\Delta_{\pi,G}$ is a character.
    \item $\Delta_{\pi,G}$ is a generalized character.
    \item $d_\pi(G)=1$.
    \item $G$ has a normal $\pi$-complement and an abelian Hall $\pi$-subgroup.
\end{enumerate}
\end{theorem}

\begin{proof}
The implication (i) $\Rightarrow$ (ii) follows by definition. For (ii) $\Rightarrow$ (iii), simply note that $[\Delta_{\pi,G},1_G] \in \mathbb{Z}$ and $0 < d_{\pi}(G) \le 1$ imply that $d_\pi(G)=1$, which is (iii). We have (iii) $\Leftrightarrow$ (iv) by \cite[Proposition 7 b)]{M14}. Finally, consider (iv) $\Rightarrow$ (i). We may write $G = H_{\pi'} \rtimes A_\pi$, where $H_{\pi'}$ is a normal Hall $\pi'$-subgroup and $A_{\pi}$ is an abelian Hall $\pi$-subgroup. Since $G$ is $\pi$-solvable, all Hall $\pi$-subgroups of $G$ are conjugate and every $\pi$-element is contained in a Hall $\pi$-subgroup by a theorem of Hall. Hence, every $\pi$-conjugacy class of $G$ has a representative in $A_\pi$. In fact, every element of $A_\pi$ is the representative of a distinct
$\pi$-conjugacy class of $G$: Otherwise, there exist distinct $a, b \in A_\pi$
with $a = b^x$ for some $x \in G$. Write $x = hc$ with $h \in H_{\pi'}$ and
$c \in A_\pi$. Since $H_{\pi'} \nrm G$, we have
$h' := c^{-1} h c \in H_{\pi'}$, and $x = ch'$. As $A_\pi$ is abelian, we
have $b^c = b$, and so $a=b^x = b^{ch'}=b^{h'}$. But then $b^{-1}a = [b, h'] \in A_\pi \cap H_{\pi'} = 1$, and so $a = b$, a contradiction. Now we calculate the multiplicity of the character $\chi \in \mathrm{Irr}(G)$ in $\Delta_{\pi,G}$ to be
\begin{align*}
    [\Delta_{\pi,G},\chi] 
    &= \frac{1}{|G|} \sum_{g \in G_\pi} \frac{|\Cen{G}{g}|\q{\chi(g)}}{|G|_\pi} 
    = \frac{1}{|G|_\pi} \sum_{g \in G_\pi} \frac{\q{\chi(g)}}{|g^G|}\\
    &= \frac{1}{|A_\pi|} \sum_{a \in A_\pi} |a^G| \frac{\q{\chi(a)}}{|a^G|}
    = \frac{1}{|A_\pi|} \sum_{a \in A_{\pi}} \q{\chi(a)} = [1_{A_\pi},\chi|_{A_\pi}],
\end{align*}
which is a nonnegative integer since it is an inner product of characters of $A_\pi$. So $\Delta_{\pi,G}$ is a character.
\end{proof}

In contrast to $\Delta_{\pi,G}$, the class function $\Pi_{\pi,G}$ is always a generalized character. The following theorem is the first part of Theorem~\ref{thmA}.

\begin{theorem}\label{prop3} 
Let $G$ be a finite group, and let $\pi$ be a collection of primes. Then $\Pi_{\pi,G}$ is a generalized character of $G$.
\end{theorem}

\begin{proof}
For every $\chi \in \mathrm{Irr}(G)$, the multiplicity
\begin{align*}
    [\Pi_{\pi,G},\chi] = \sum_{g \in G_\pi} \frac{|\Cen{G}{g}|\q{\chi(g)}}{|G|} = \sum_{g \in G_\pi} \frac{\q{\chi(g)}}{|g^G|} = \sum_{g^G \in \mr{Cl}_\pi(G)} \q{\chi(g)}
\end{align*}
is an algebraic integer. Since the Galois group of the cyclotomic field $\mathbb{Q}_{|G|}$ permutes the conjugacy classes of $G$ (preserving the orders of elements), the multiplicity $[\Pi_{\pi,G},\chi]$ is fixed by every such Galois automorphism, so it is rational. Hence, $[\Pi_{\pi,G},\chi]$ is an integer and $\Pi_{\pi,G}$ is a generalized character.
\end{proof}

Since the class function $\Pi_{\pi,G}$ is always a generalized character, it is then natural to ask when it is a character. Our next theorem, which was stated as the second claim in Theorem~\ref{thmA} of the Introduction, shows that $\Pi_{\pi,G}$ is a character in two important special cases. As mentioned there, both cases were obtained independently by Robinson~\cite[Remark 5.1 and
Theorem~8.1]{Ro25}; we therefore only sketch the proof in the $\pi$-separable case, where our argument is very similar to his.

\begin{theorem}\label{prop4} 
Let $\pi$ be a collection of primes. If 
\begin{enumerate}[nolistsep,label=(\roman*)]
    \item $\pi = p'$ for some prime $p$ and $G$ is an arbitrary finite group or
    \item $\pi$ is arbitrary and $G$ is a $\pi$-separable finite group,
\end{enumerate}
then $\Pi_{\pi,G}$ is a character of $G$. 
\end{theorem}
\begin{proof}
(i) Let $G$ be an arbitrary finite group and $\pi=p'$ for some prime $p$. For every $\chi \in \mathrm{Irr}(G)$, define the decomposition numbers $d_{\chi,\varphi}$ by $\chi^0 = \sum_{\varphi \in \mathrm{IBr}_p(G)} d_{\chi,\varphi} \varphi$, where $\chi^0$ is the restriction of $\chi$ to the $p$-regular elements of $G$. The decomposition numbers are nonnegative integers since $\chi^0$ is a Brauer character of $G$ by \cite[Corollary~2.9]{N98}. Then for every $\varphi \in \mathrm{IBr}_p(G)$ we may define the projective indecomposable character $\Phi_\varphi$ associated to $\varphi$ by $\Phi_\varphi = \sum_{\chi \in \mathrm{Irr}(G)} d_{\chi,\varphi} \chi$. Note that $\Phi_\varphi$ is an ordinary character of $G$ since it is a nonnegative integral linear combination of characters of $G$. Now, by \cite[Theorem 2.13]{N98}, the set $\{\Phi_\varphi \mid \varphi \in \mathrm{IBr}_p(G)\}$ is a basis for those class functions of $G$ which vanish off of $p$-regular elements. Since $\Pi_{p',G}$ vanishes off of $p$-regular elements by definition, we have
\begin{align*}
    \Pi_{p',G} = \sum_{\varphi \in \mathrm{IBr}_p(G)} e_\varphi \Phi_\varphi
\end{align*}
for some $e_\varphi \in \mathbb{C}$. Now, it follows from \cite[Theorem 2.13]{N98} that
\begin{align*}
    \q{e_\varphi} = [\varphi,\Pi_{p',G}]^0 := \frac{1}{|G|} \sum_{g \in G_{p'}} \varphi(g) \q{\Pi_{p',G}(g)}.
\end{align*}
In addition, it follows from \cite[Corollary~2.17]{N98} that \mbox{$e_\varphi = \q{e_\varphi} \in \mathbb{Z}$} since $\Pi_{p',G}$ is a generalized character vanishing off of $p$-regular elements by Theorem~\ref{prop3}. (Indeed, \cite[Corollary~2.17]{N98} shows that $[\varphi,\chi]^0$
is an integer when $\chi$ is any ordinary character vanishing off of $p$-regular elements, and the same proof goes through for generalized characters vanishing off of $p$-regular elements.)

Thus, in order to prove that $\Pi_{p',G}$ is a character, it is enough to show that $e_\varphi$ is nonnegative, since then $\Pi_{p',G}$ is a nonnegative integral sum of characters. Using the definition of $\Pi_{p',G}$, we have
\begin{align*}
    e_\varphi = \frac{1}{|G|} \sum_{g \in G_{p'}} |\mathbf{C}_G(g)| \varphi(g) = \sum_{g \in G_{p'}} \frac{\varphi(g)}{|g^G|} = \sum_{g^G \in \mr{Cl}_{p'}(G)} \varphi(g).
\end{align*}
That is, $e_\varphi$ is the sum of entries in the row of the Brauer character table corresponding to~$\varphi$. Such sums have been investigated in \cite{chen21,chillag98}. In particular, it is remarked in the proof of \cite[Theorem 1.1]{chen21} that \cite[Example 1.8 and Proposition~2.2(b)]{chillag98} proves that the row sums in the Brauer character table are nonnegative real numbers. Hence, $e_\varphi$ is a nonnegative integer and $\Pi_{p',G}$ is a character of $G$. 

(ii) Now, if $\pi$ is an arbitrary collection of primes and $G$ is $\pi$-separable, then an exactly analogous proof goes through. Namely, in place of the Brauer characters we have the so-called $\pi$-partial characters, i.e. the restrictions of ordinary characters of $G$ to the $\pi$-elements of~$G$. The collection $I_\pi(G)$ of irreducible $\pi$-partial characters consists of those $\pi$-partial characters that cannot be written as a sum of two $\pi$-partial characters. Then, in the same way as before, we may define decomposition numbers for each $\pi$-partial character $\chi^0$ with $\chi \in \Irr{G}$, which are necessarily nonnegative integers by the construction of $\pi$-partial characters: The decomposition numbers simply reflect the decomposition of $\chi^0$ into the linearly independent irreducible $\pi$-partial characters. Then we define projective indecomposable characters $\Phi_\varphi$ for $\varphi \in \mathrm{I}_\pi(G)$. It is proved in \cite[Corollary 3.8]{I18} that the characters $\Phi_\varphi$ constitute a basis for class functions vanishing off of $\pi$-elements. Again, the multiplicity of $\Phi_\varphi$ in $\Pi_{\pi,G}$ is an integer by~\cite[Problem~3.4]{I18} since $\Pi_{\pi,G}$ is a generalized character by Theorem~\ref{prop3}; as before, the result is stated there for characters, but the same proof goes through for generalized characters. The multiplicity is equal to a row sum in the $\pi$-partial character table of $G$ by~\cite[Problem~3.5]{I18}. It is remarked in the proof of \cite[Theorem~2.4]{chen21} that \cite[Proposition~2.2(b)]{chillag98} proves that the row sums in the $\pi$-partial character table are nonnegative real numbers. So the multiplicities must be nonnegative integers, and again $\Pi_{\pi,G}$ is a character of $G$.
\end{proof}

\begin{remark}\label{r:rowsums}
One can show that the row sums in the ordinary character table are nonnegative integers using the conjugating character~\cite[Problem~5.13]{I94}. In a nicely parallel result, the proof of Theorem \ref{prop4} shows that the row sums in the $p$-Brauer character table are nonnegative integers by using the ``local'' conjugating character~$\Pi_{p',G}$. This result sharpens \cite[Proposition~2.2(b)]{chillag98}, where it was shown that row sums in the $p$-Brauer table are nonnegative real numbers.
\end{remark}

Supported by Theorem~\ref{prop4} and computational evidence, we state the following conjecture, which would answer a question of Robinson in~\cite[Section~8]{Ro25}.

\begin{conjecture}\label{c:Picharacter}
Let $G$ be a finite group, and let $\pi$ be an arbitrary collection of primes. Then $\Pi_{\pi,G}$ is a character of $G$; that is, $[\Pi_{\pi,G},\chi] \ge 0$ for all $\chi \in \Irr{G}$. 
\end{conjecture}

A counterexample to Conjecture~\ref{c:Picharacter} would consist of a finite group $G$, a collection of primes~$\pi$ and a character $\chi \in \Irr{G}$ such that $[\Pi_{\pi,G},\chi]<0$. A natural first place to look for such a counterexample is among those groups $G$ and characters $\chi \in \Irr{G}$ whose multiplicities in the conjugating character $\Pi_G$ are as small as possible, namely \mbox{$[\Pi_G,\chi]=0$}; the idea being that one could perhaps make a clever choice of $\pi$ to force the multiplicity $[\Pi_{\pi,G},\chi]$ to be negative. If $G$ is a nonabelian simple group, the collection of such pairs $(G,\chi)$ is remarkably small: $[\Pi_G,\chi]=0$ if and only if $G \cong \mr{PSU}_n(q)$ ($n \ge 3$) with $n$ coprime to $2(q+1)$ and $\chi$ is the Weil character of degree $(q^n-q)/(q+1)$~\cite{heide13, heide06}. In our next theorem, we show that $\mr{PSU}_3(q)$ with $(3,q+1)=1$ and its Weil character $\chi$ of degree $q^2-q$ do not yield a counterexample to Conjecture~\ref{c:Picharacter} for any collection $\pi$ of primes.

Note that Conjecture~\ref{c:Picharacter} is equivalent to the statement that for any finite group and any collection $\pi$ of primes, the partial row sum over $\pi$-conjugacy classes in any row of the character table is a nonnegative integer, since for any $\chi \in \Irr{G}$,
\begin{align*}
    [\Pi_{\pi,G},\chi] = \sum_{g \in G_\pi} \frac{|\Cen{G}{g}|\q{\chi(g)}}{|G|} = \sum_{g \in G_\pi} \frac{\q{\chi(g)}}{|g^G|} = \sum_{g^G \in \mr{Cl}_\pi(G)} \q{\chi(g)} = \sum_{g^G \in \mr{Cl}_\pi(G)} \chi(g),
\end{align*}
where the last equality follows since $\Pi_{\pi,G}$ is a generalized character by Theorem~\ref{prop3}. This is the statement that we prove in Theorem~\ref{p:weil} for $\mr{PSU}_3(q)$ and $(3,q+1)=1$, which was stated as Theorem~\ref{thmB} in the Introduction. We remark that this is not the approach taken in \cite{heide13,heide06} to investigate the multiplicities $[\Pi_G,\chi]$ for $\chi \in \Irr{G}$, as the authors take advantage of the fact that $\Pi_G$ decomposes into a sum of permutation characters on its conjugacy classes. (See \cite[Lemma~2.1]{heide13} for the consequences of this observation.) As we do not have such an interpretation of $\Pi_{\pi,G}$ -- in particular, $\Pi_{\pi,G}$ is not the sum of permutation characters on $\pi$-conjugacy classes, the latter being a character which takes the value $|\Cen{G}{g}_\pi|$ on each element $g \in G$ -- we consider the row sum of the character table directly. 
 
\begin{theorem}\label{p:weil}
If $G \cong \mr{PSU}_3(q)$ with $(3,q+1)=1$ and $\chi$ is the Weil character of degree $q^2-q$, then $[\Pi_{\pi,G},\chi]\ge 0$ for all collections $\pi$ of primes.
\end{theorem}

\begin{proof}
To begin, we first calculate in GAP~\cite{GAP} that $[\Pi_{\pi,G},\chi] \ge 0$ for all collections $\pi$ of primes and \mbox{$q=3,4,7,9$}; these are the prime powers $q \le 9$ with $(3,q+1)=1$. So we may assume that $q>9$. The character table of $\mr{PSU}_3(q)$ is given in \cite[Table~2]{Simpson73}, and the relevant information is summarized in Table~\ref{t:psu3}. We omit information about classes on which $\chi$ vanishes, as they will not contribute to the sum over conjugacy classes that calculates the multiplicity $[\Pi_{\pi,G},\chi]$. In Table~\ref{t:psu3}, $C_i$ labels the type of the conjugacy class as defined in \cite[Table~2]{Simpson73} whose elements are $\pi_i$-elements, $N_i$ is the number of conjugacy classes of type $C_i$, and $\chi(C_i)$ is the value of $\chi$ at any element of $C_i$. Note that elements in distinct members of the family of classes $C_i$ need not have the same order, but just the same character value.

\begin{table}[ht]
    \centering
    \begin{tabular}{cccc}
    $C_i$ & $\pi_i$ & $N_i$ & $\chi(C_i)$  \\ \hline
    $C_1$ & $\emptyset$ & 1 & $q^2-q$  \\
    $C_2$ & $\pi(q)$ & 1 & $-q$  \\
    $C_3$ & -- & -- & $0$  \\
    $C_4$ & $\pi(q+1)$ & $q$ & $-(q-1)$  \\
    $C_5$ & $\pi(q(q+1))$ & $q$ & $1$  \\
    $C_6$ & $\pi(q+1)$ & $\frac16(q^2-q)$ & 2  \\
    $C_7$ & -- & -- & 0  \\
    $C_8$ & $\pi(q^2-q+1)$ & $\frac13(q^2-q)$ & $-1$ \\
    \end{tabular}
    \caption{Character table for $\mr{PSU}_3(q)$ with $(3,q+1)=1$.}
    \label{t:psu3}
\end{table}
Assume for contradiction that there is some $q$ and collection $\pi$ of primes such that $[\Pi_{\pi,G},\chi]<0$; that is, the sum over $\pi$-conjugacy classes in the row of the character table corresponding to $\chi$ is negative. Any collection $\pi$ admits the contribution from $N_1\chi(C_1)=q^2-q$. Now we use the character table to consider which primes are forced to belong to such a  $\pi$. Note that $(q+1,q^2-q+1)=(q+1,3)=1$ since $q^2-q+1 = (q+1)(q-2)+3$, so adjoining any prime dividing $q(q+1)$ to $\pi$ introduces no conjugacy classes of type $C_8$.

We first show that $\pi(q+1) \subseteq \pi$. The only negative contributions come from class types $C_2$, $C_4$ and $C_8$.  Note that we must include some classes of type $C_4$ because $N_1 \chi(C_1) + N_2 \chi(C_2) + N_8 \chi(C_8) = \frac{2}{3}q^2 - \frac53q$, which is greater than zero as $q>2$. Since $\chi(C_4)=-(q-1)$, the number of classes of type $C_4$ that we must include to obtain a negative multiplicity is bounded below by
\begin{align*}
    \frac{\frac{2}{3}q^2-\frac{5}{3}q}{q-1} = \frac23 q - \frac{q}{q-1} > \frac12 (q+1) \qquad \text{for } q>9.
\end{align*}
By \cite[Table~2]{Simpson73}, each class of type $C_4$ has a representative in a cyclic subgroup $T$ of order~$q+1$, namely $T=\{\mr{diag}(a,a,a^{-2}) \mid a^{q+1}=1\}$, contained in a maximal torus of order~$(q+1)^2$, and distinct nonidentity elements of $T$ lie in distinct classes. Therefore, given any subset $\tau \subseteq \pi(q+1)$ that we choose to include in our~$\pi$, the identity and the representatives that we obtain in a class of type $C_4$ constitute a maximal cyclic $\tau$-subgroup of $T$. The above calculation shows that we must include more than half the elements of $T$. So by Lagrange's Theorem, all the elements of $T$ must be included. This means all of $\pi(q+1)$ is included in $\pi$, which is what we wanted to show.

Therefore, $[\Pi_{\pi,G},\chi]$ includes the contribution $N_4 \chi(C_4) = -(q^2-q)$. But since $\pi_6=\pi(q+1)$ as well, we also have the contribution $N_6\chi(C_6)=+\frac13(q^2-q)$. So if $\pi=\pi(q+1)$, we have $[\Pi_{\pi,G},\chi]=\frac13 (q^2-q)$, and we need to add more primes if we are to obtain a negative multiplicity.

Next, notice that if we add the defining prime $\pi(q)$ to $\pi$, then since $\pi(q+1) \subseteq \pi$, we obtain the contributions from classes of types $C_2$ and $C_5$, which exactly cancel. Therefore, it is inconsequential whether we include $\pi(q)$ in $\pi$ or not, as in both cases we are left with $[\Pi_{\pi,G},\chi]=\frac13(q^2-q)$ at this point.

Our last remaining choice is whether to add some subset of $\pi(q^2-q+1)$ to $\pi$. The contribution from all classes of type $C_8$ is $N_8\chi(C_8)=-\frac13(q^2-q)$, so regardless of which remaining primes we add to $\pi$, we always have $[\Pi_{\pi,G},\chi] \ge 0$. This contradiction concludes the proof.
\end{proof}

Having addressed the issue of $\Delta_{\pi,G}$ and $\Pi_{\pi,G}$ being (generalized) characters, we now turn to the decomposition of $\Delta_{\pi,G}$ and $\Pi_{\pi,G}$. This is Question~(2) of the introduction. The next theorem characterizes when all but one of the multiplicities vanish, in which case the class function is proportional to a single irreducible character. 

\begin{theorem}\label{prop5} 
Let $G$ be a finite group, and let $\pi$ be a collection of primes. Then $\Pi_{\pi,G}$, or equivalently $\Delta_{\pi,G}$, has only one irreducible constituent if and only if $G$ is an abelian~$\pi$-group.
\end{theorem}

\begin{proof}
The key here is to simply recall that $[\Pi_{\pi,G},1_G]=k_\pi(G)>0$, so $\Pi_{\pi,G}$ has one constituent if and only if it is proportional to the trivial character. In particular, $\Pi_{\pi,G}$ cannot vanish on any element of $G$, so $G$ is a $\pi$-group by the definition of~$\Pi_{\pi,G}$. Now, $|G|=\Pi_{\pi,G}(1)$, so $\Pi_{\pi,G}=|G| 1_G$. Finally, $|G|=[\Pi_{\pi,G},1_G]=k_\pi(G)=k(G)$, so $G$ is abelian. The reverse implication is obvious because $\Pi_{\pi,G}=|G|1_G$ under the assumptions of the theorem. 
Since $\Delta_{\pi,G}$ is a positive scalar multiple of $\Pi_{\pi,G}$, the two have identical constituent sets.
\end{proof}

Theorem~\ref{prop5} shows that $\Pi_{\pi,G}$ has one constituent if and only if $G$ is an abelian $\pi$-group. We might then suspect that $\Pi_{\pi,G}$ has many constituents when $G$ is not an abelian $\pi$-group. As mentioned above, it is well-known that the conjugating character $\Pi_G$ has many constituents when $G$ is nonabelian simple. If $G$ is a $\pi'$-group, then $\Pi_{\pi,G}$ is the regular character and contains every irreducible character of $G$ as a constituent, so the interesting cases are when $\pi$ is a nonempty proper subset of the prime divisors of $|G|$.  Decomposing $\Pi_{\pi,G}$ into its constituents seems difficult in general, but it could be interesting to study its decomposition for some particular classes of groups and values of $\pi$. For example, it is known that $[\Pi_{2',S_n},\chi]>0$ for all $\chi \in \Irr{S_n}$ and $n \ge 1$~\cite[Theorem~4.9]{Su18}; this is more than what is implied by Theorem~\ref{prop4}(i) with $p=2$.

\section{Characterizing group structure}\label{s:Q3}

In this section, we turn to Question~(3) of the introduction, which asks to characterize group structure in terms of the behavior of a class function. We can state our first result for the more general class functions of Example~\ref{ex:X}. Recall that if $X \subseteq G$ is any subset closed under conjugation, then we define $\Pi_{X,G}(g)=|\Cen{G}{g}|$ if $g \in X$ and $0$ otherwise. We write $\mr{Cl}_X(G)$ to denote the set of $G$-conjugacy classes contained in $X$. Letting $X=G_\pi$, the set of $\pi$-elements of $G$, recovers $\Pi_{\pi,G}$ from the last section.

\begin{theorem}\label{t:Xclasses}
Let $H$ be a subgroup of a finite group $G$. Let $X \subseteq H$ be a subset closed under conjugation by $H$, and let $Y \subseteq G$ be a subset closed under conjugation by $G$. The following are equivalent:
\begin{enumerate}[nolistsep,label=(\roman*)]
    \item $(\Pi_{X,H})^G = \Pi_{Y,G}$.
    \item The map $h^H \mapsto h^G$ is a bijection from $\mr{Cl}_X(H)$ to $\mr{Cl}_Y(G)$.
\end{enumerate}
\end{theorem}

\begin{proof}
Letting $g\in G$, we begin by deriving a formula for $(\Pi_{X,H})^G(g)$. The set $g^G \cap X$ is fixed setwise under conjugation by elements of $H$, so it is equal to a (possibly empty) union $h_1^H \cup \cdots \cup h_m^H$ of $H$-conjugacy classes in $X$. If $g^G \cap X$ is nonempty, then using the definition of an induced character (see, in particular, \cite[p.~64]{I94}), we have that
\begin{align*}
    (\Pi_{X,H})^G(g) = |\Cen{G}{g}| \sum_{i=1}^m \frac{\Pi_{X,H}(h_i)}{|\Cen{H}{h_i}|} = m |\Cen{G}{g}|.
\end{align*}
We write this as
\begin{align}\label{eq:ind}
    (\Pi_{X,H})^G(g) = \begin{cases}
    m |\Cen{G}{g}| & \text{if } g^G \cap X \text{ is a union of $m\ge1$ $H$-classes,} \\
    0 & \text{if } g^G \cap X = \emptyset.
    \end{cases}
    \tag{$*$}
\end{align}

Now we continue with the proof of the theorem.

(ii) $\Rightarrow$ (i): First note that $X \subseteq Y$ since, if $h \in X$, then $h^G \in \mr{Cl}_Y(G)$ by assumption, and so $h \in Y$. Now let $g \in G$. If $g \not\in Y$, then $\Pi_{Y,G}(g)=0$ by definition. Since $X \subseteq Y$, $g^G \cap X \subseteq g^G \cap Y = \emptyset$, and so $(\Pi_{X,H})^G(g) = 0$ as well by Equation~(\ref{eq:ind}). Now assume that $g \in Y$. By the bijection in (ii), $g^G \cap X$ is exactly one $H$-class: surjectivity gives some $h \in H$ with $h \in g^G \cap X$, and if two distinct $H$-classes sat inside $g^G \cap X$, they would have the same image, contradicting injectivity. Therefore, $m=1$ in~(\ref{eq:ind}), and so $(\Pi_{X,H})^G(g)=|\Cen{G}{g}|=\Pi_{Y,G}(g)$.

(i) $\Rightarrow$ (ii): Let $h \in X$. Then of course $h^G \cap X \neq \emptyset$, and so $(\Pi_{X,H})^G(h) \neq 0$ by (\ref{eq:ind}). Since $(\Pi_{X,H})^G(h) = \Pi_{Y,G}(h)$, this means, in particular, that $h \in Y$. So again $X \subseteq Y$. Therefore, $h^H \mapsto h^G$ is indeed a map $\mr{Cl}_X(H) \rightarrow \mr{Cl}_Y(G)$.

The map is injective: Let $h_1^H,h_2^H \in \mr{Cl}_X(H)$ such that $h_1^G=h_2^G \in \mr{Cl}_Y(G)$. By assumption, $(\Pi_{X,H})^G(h_1) = \Pi_{Y,G}(h_1)$, which means that $h_1^G \cap X$ is a single $H$-class, namely~$h_1^H$, by (\ref{eq:ind}). In the same way, $h_2^G \cap X = h_2^H$. Therefore, \mbox{$h_1^H = h_1^G \cap X = h_2^G \cap X = h_2^H$.}

The map is surjective: 
Let $g^G \in \mr{Cl}_Y(G)$. By assumption, $(\Pi_{X,H})^G(g) = \Pi_{Y,G}(g)$, which means that $g^G \cap X$ is exactly one $H$-class lying in $X$, say $h^H \in \mr{Cl}_X(H)$, by (\ref{eq:ind}). Hence, $h^H \mapsto h^G = g^G$.
\end{proof}

We note that in all of our applications of Theorem~\ref{t:Xclasses}, $X=Y \cap H$. 

Theorem~\ref{t:Xclasses} characterizes when $(\Pi_{X,H})^G = \Pi_{Y,G}$ group-theoretically. We are also interested in a representation-theoretic characterization of such equalities. As our starting point, consider the following restatement of a theorem of Sangroniz.

\begin{theorem}{\cite[Theorem~3.3]{S08}}\label{t:sangronitz}
Let $p$ be a prime, and let $H$ be a subgroup of a finite $p$-solvable group $G$. Then the following are equivalent:
\begin{enumerate}[nolistsep,label=(\roman*)]
       \item $d_{p'}(G)=d_{p'}(H)$.
       \item Restriction is a surjection from $\mr{IBr}_p(G)$ to $\mr{IBr}_p(H)$.
\end{enumerate}
\end{theorem}
\begin{proof}
Taking $\pi=p'$ in \cite[Theorem~3.3]{S08} for the $p$-solvable group $G$, $d_{p'}(G) = d_{p'}(H)$ if and only if all irreducible Brauer characters of $G$ restrict irreducibly to $H$. Since every irreducible Brauer character of $H$ is a constituent of the restriction of some irreducible Brauer character of $G$ by \cite[Corollary~8.3]{N98}, this is equivalent to restriction being a surjection $\mr{IBr}_p(G) \rightarrow \mr{IBr}_p(H)$.
\end{proof}

Theorem~\ref{t:sangronitz} cannot be extended in either direction when $G$ is not $p$-solvable; see the examples due to Navarro at the end of \cite{S24}.

Since $(\Delta_{p',H})^G = \Delta_{p',G}$ implies that $d_{p'}(H) = d_{p'}(G)$ by Frobenius Reciprocity, the representation-theoretic characterization of $(\Delta_{p',H})^G = \Delta_{p',G}$ in the next result is a refinement of Theorem~\ref{t:sangronitz}. We show that it is a strict refinement in Example~\ref{ex:refine}. Parts (ii) and (iii) of the next theorem were stated as Theorem~\ref{thmD} in the Introduction.

\begin{theorem}\label{t:resbij}
Let $p$ be a prime, and let $H$ be a subgroup of a finite group~$G$.

(i)  If restriction is a bijection from $\mr{IBr}_p(G)$ to $\mr{IBr}_p(H)$, then $\Pi_{p',G}=(\Pi_{p',H})^G$.

(ii) If restriction is a bijection from $\mr{IBr}_p(G)$ to $\mr{IBr}_p(H)$ and $[G:H]_{p'}=1$, then $\Delta_{p',G}=(\Delta_{p',H})^G$.

(iii) If $G$ is $p$-solvable, then the converse of (ii) holds. 

\end{theorem}

\begin{proof}
(i) We first show that if restriction is a bijection $\mr{IBr}_p(G) \rightarrow \mr{IBr}_p(H)$, then \mbox{$h^H \mapsto h^G$} is a bijection \mbox{$\mr{Cl}_{p'}(H) \rightarrow \mr{Cl}_{p'}(G)$}. This shows that $\Pi_{p',G}=(\Pi_{p',H})^G$ by Theorem~\ref{t:Xclasses}. It is clearly a well-defined map. Since $|\mr{IBr}_p(G)|=|\mr{Cl}_{p'}(G)|$, the bijection implies that $|\mr{Cl}_{p'}(G)| = |\mr{Cl}_{p'}(H)|$. So we need only show that $h^H \mapsto h^G$ is an injective map $\mr{Cl}_{p'}(H) \rightarrow \mr{Cl}_{p'}(G)$; that is, $h^G \cap H$ is at most one $H$-class. If not, there exists some irreducible Brauer character $\psi \in \mr{IBr}_p(H)$ that takes different values on the distinct $H$-classes $h_1^H,h_2^H \subseteq h^G \cap H$ since $\mr{IBr}_p(H)$ spans the vector space of complex-valued class functions on $H_{p'}$. That is, $\psi(h_1) \neq \psi(h_2)$. Since restriction $\mr{IBr}_p(G) \rightarrow \mr{IBr}_p(H)$ is a bijection, $\psi=\varphi_H$ for some $\varphi \in \mr{IBr}_p(G)$. But $\varphi$ is a class function on~$G$, so $\psi(h_1) = \varphi_H(h_1) = \varphi_H(h_2) = \psi(h_2)$, a contradiction.

(ii) By (i), $\Pi_{p',G}=(\Pi_{p',H})^G$, and $[G:H]_{p'}=1$ means that we may divide through by $|G|_{p'}=|H|_{p'}$ to conclude that $\Delta_{p',G}=(\Delta_{p',H})^G$. 

(iii) First note that since $\Delta_{p',G}=(1/|G|_{p'}) \Pi_{p',G}$ by definition, Equation~(\ref{eq:ind}) in the proof of Theorem~\ref{t:Xclasses} implies that
\begin{align*}\label{eq:delta}
    (\Delta_{p',H})^G(g)=
    \begin{cases}
        m\dfrac{|\Cen{G}{g}|}{|H|_{p'}} & \text{if $g^G \cap H_{p'}$ is a union of $m\ge 1$ $H$-classes} \\
        0 & \text{if $g^G \cap H_{p'} = \emptyset$}
    \end{cases}
\end{align*}
for all $g \in G$. Note that we have taken $Y=G_{p'}$ and $X=Y \cap H = H_{p'}$ in Theorem~\ref{t:Xclasses}. In particular, we have $(\Delta_{p',H})^G = \Delta_{p',G}$ if and only if for every $p$-regular $g \in G$
\begin{align*}
    \frac{|\Cen{G}{g}|}{|G|_{p'}} = \Delta_{p',G}(g) = (\Delta_{p',H})^G(g) = m \frac{|\Cen{G}{g}|}{|H|_{p'}}
\end{align*}
and $(\Delta_{p',H})^G(g) = 0 = \Delta_{p',G}(g)$ for all $p$-singular $g \in G$. Summarizing the situation implied by this equation, we have
\begin{align*}
    &\Delta_{p',G} = (\Delta_{p',H})^G  \\ 
    &\iff
    \text{$[G:H]_{p'}=1$ and $m=1$ for every $g^G \in \mr{Cl}_{p'}(G)$} \\
    &\iff
    [G:H]_{p'}=1 \text{ and }(\Pi_{p',H})^G = \Pi_{p',G}.
\end{align*}
As we have mentioned before, $\Delta_{p',G}=(\Delta_{p',H})^G$ implies that $d_{p'}(G)=d_{p'}(H)$. So restriction is a surjection $\mr{IBr}_p(G) \rightarrow \mr{IBr}_p(H)$ by Theorem~\ref{t:sangronitz} since $G$ is $p$-solvable. By the equivalences above, our hypothesis implies that $[G:H]_{p'}=1$ and $(\Pi_{p',H})^G = \Pi_{p',G}$. Therefore, $|\mr{IBr}_p(H)| = |\mr{Cl}_{p'}(H)|=|\mr{Cl}_{p'}(G)| = |\mr{IBr}_p(G)|$ by Theorem~\ref{t:Xclasses}, and so restriction is a bijection.
\end{proof}

The following example, identified by Navarro~\cite{N17}, shows that the converse of Theorem~\ref{t:resbij}(i) fails. Let $G$ be the solvable Frobenius group $(C_3 \times C_3) \rtimes C_8$ of order $72$ with $C_8$ permuting the $8$ elements of order $3$ in $C_3 \times C_3$; this is  \texttt{SmallGroup(72,39)} in GAP. Let $H \le G$ be isomorphic to $\mr{S}_3$, and let $p=2$. Then $[G:H]_{2'}=3$. Both $G$ and $H$ have just one nonidentity class of $2$-regular elements, so $h^H \mapsto h^G$ is a bijection $\mr{Cl}_{2'}(H) \rightarrow \mr{Cl}_{2'}(G)$. This means $(\Pi_{2',H})^G=\Pi_{2',G}$ by Theorem~\ref{t:Xclasses}. However, the only nonlinear degree of an irreducible $2$-Brauer character of $G$ is $8$, and that of $H$ is~$2$; so restriction is not even a well-defined map $\mr{IBr}_2(G) \rightarrow \mr{IBr}_2(H)$, and in particular it is not a bijection. 

The next example shows that Theorem~\ref{t:resbij}(ii),(iii) is a strict refinement of Theorem~\ref{t:sangronitz}.

\begin{example}\label{ex:refine}
We show that there exists a $p$-solvable group $G$ and subgroup \mbox{$H \le G$} such that $d_{p'}(G)=d_{p'}(H)$ but $\Delta_{p',G} \neq (\Delta_{p',H})^G$; that is, the invariant $d_{p'}(G)$ cannot distinguish between $G$ and $H$, but $\Delta_{p',G}$ can. So $\Delta_{p',G}$ accomplishes the task of distinguishing groups by their ``timbre'' as set out in the introduction.

In fact, if $G$ is $p$-solvable and $d_{p'}(G)=d_{p'}(H)$, then $\Delta_{p',G} = (\Delta_{p',H})^G$ if and only if $[G:H]_{p'}=1$: The forward direction follows directly from Theorem~\ref{t:resbij}(iii). For the reverse direction, $d_{p'}(G)=d_{p'}(H)$ implies restriction is a surjection by Theorem~\ref{t:sangronitz}, and $[G:H]_{p'}=1$ implies furthermore that $k_{p'}(G)=k_{p'}(H)$. So restriction is a bijection, and we may apply Theorem~\ref{t:resbij}(ii). 

Let $G = \mr{S}_4$, $H = \mr{A}_4$ and $p=3$. The group $G$ is solvable, and so it is $3$-solvable. Then $d_{3'}(G)=1/2 = d_{3'}(H)$, but $[G:H]_{3'}=2 > 1$, so $\Delta_{3',G} \neq(\Delta_{3',H})^G$ by Theorem~\ref{t:resbij}(iii). More explicitly, they differ at the identity element, for example: $\Delta_{3',G}(1)=|G|/|G|_{3'}=3$ and $(\Delta_{3',H})^G(1)=[G:H]\cdot |H|/|H|_{3'}=6$. Regarding the different ``timbres'' of $G$ and~$H$, we have for the sign character $\epsilon \in \Irr{\mr{S}_4}$ that 
\begin{align*}
    [\Delta_{3',G},\epsilon] =\frac18( \epsilon(1) + \epsilon((12)(34)) + \epsilon((12)) + \epsilon((1234))) = \frac18 (1+1-1-1)=0,
\end{align*}
where we recall that the multiplicity $[\Pi_{3',G},\epsilon]$ is simply the sum over the $3$-regular classes of $G$ in the row of the character table corresponding to~$\epsilon$ and $\Delta_{3',G} = (1/|G|_{3'})\Pi_{3',G}$. On the other hand, 
\begin{align*}
    [(\Delta_{3',H})^G,\epsilon] = [\Delta_{3',H},\epsilon|_H]= [\Delta_{3',H},1_H] = d_{3'}(H) = \frac12.
\end{align*}
So the ``fundamental frequency'' (i.e. the multiplicity of $1_G$) does not distinguish $G$ and~$H$, but a ``higher harmonic'' (i.e. the multiplicity of $\epsilon$) does distinguish them.
\end{example}

Our next result implies that Theorem~\ref{t:resbij} is trivial if $G$ is a $p'$-group:

\begin{proposition}
Let $p$ be a prime, and let $H$ be a subgroup of an arbitrary finite group~$G$. Suppose that $|G|$ is not divisible by $p$. Then any of the following statements imply that $G=H$:
\begin{enumerate}[nolistsep,label=(\roman*)]
    \item $\Delta_{p',G}=(\Delta_{p',H})^G$.
    \item $\Pi_{p',G}=(\Pi_{p',H})^G$.
    \item The map $h^H \mapsto h^G$ is a bijection $\mr{Cl}_{p'}(H) \rightarrow \mr{Cl}_{p'}(G)$.
    \item Restriction is a bijection $\mr{IBr}_p(G) \rightarrow \mr{IBr}_p(H)$.
\end{enumerate}
\end{proposition}
\begin{proof}
If (i) holds, then $1=[G:H]_{p'}=[G:H]$ by Theorem~\ref{t:resbij}(iii) since $G$ is $p$-solvable. Now assume (iii), which is equivalent to (ii) by Theorem~\ref{t:Xclasses}. Let $g \in G$ be an arbitrary element, and assume for contradiction that $H<G$. Since $\mr{Cl}_{p'}(G)=\mr{Cl}(G)$, $g^G \cap H$ is nonempty by assumption. Hence, $g$ is contained in some $G$-conjugate of $H$. But then $G$ is covered by the conjugates of a proper subgroup, a contradiction. Finally, assume~(iv). Then $\mr{IBr}_p(G)=\Irr{G}$ by \cite[Theorem~2.12]{N98}, and the bijection implies that $|G| = \sum_{\chi \in \Irr{G}} \chi_H(1)^2 = |H|$.
\end{proof}

To conclude this section, we show that the class function $\Pi_{X,G}$ is always a character when $X$ is the collection of real elements of $G$. This was stated as Theorem~\ref{thmC} in the Introduction.

\begin{theorem}\label{t:real}
Let $G$ be a finite group, and let $R \subseteq G$ be the collection of real elements of~$G$. Then $\Pi_{R,G}$ is a character of $G$.
\end{theorem}
\begin{proof}
Recall that the second orthogonality relation \cite[Theorem~2.18]{I94} applied to the elements $g, g^{-1} \in G$ states that
\begin{align*}
    \sum_{\chi \in \Irr{G}} \chi(g) \q{\chi(g^{-1})}= \begin{cases}
        |\Cen{G}{g}| & \text{if $g$ is conjugate to $g^{-1}$} \\
        0 & \text{otherwise}.
    \end{cases}
    \;=\; \Pi_{R,G}(g).
\end{align*}
Therefore, $\Pi_{R,G}(g) = \sum_{\chi \in \Irr{G}} \chi^2(g)$. Since products of characters are characters (see \cite[Corollary~4.2]{I94}), $\Pi_{R,G}$ is also a character.
\end{proof}

\section{A dual construction}\label{s:dual}
Up to now, we have extended a group invariant $I_G$ to a class function $\mathbf{I}_G$ by encoding the original invariant $I_G$ in the sum of values normalized by $|G|$; that is, $(1/|G|) \sum_{g \in G} \mbf{I}_G(g) = [\mbf{I}_G, 1_G] = I_G$. We now consider a dual construction in which an invariant $I_G$ is encoded in a class function $\mbf{I}_G$ as the sum of multiplicities of the irreducible characters normalized by $|G|$, namely $I_G = (1/|G|) \sum_{\chi \in \Irr{G}} [\mbf{I}_G,\chi]$. We consider one interesting example of this dual construction. For a prime $p$, define the class function (character) $\mbf{D}_{p'}(G) = \sum_{p \nmid \chi(1)} \chi$, where the sum is taken over all irreducible characters of a finite group $G$ whose degree is not divisible by $p$. This gives rise to an invariant
\begin{align*}
    D_{p'}(G) = \frac{1}{|G|} \sum_{\chi \in \Irr{G}} [\mathbf{D}_{p'}(G),\chi] = \frac{|\mathrm{Irr}_{p'}(G)|}{|G|}.
\end{align*}

We first prove a bound for $D_{p'}(G)$. In the proof, we use the recently-proved McKay Conjecture~\cite{cab24}.

\begin{theorem}\label{t:Pabelian}
Let $G$ be a finite group, let $p$ be a prime, and let $P$ be a Sylow $p$-subgroup of $G$. By the Schur--Zassenhaus Theorem, we may write $\mathbf{N}_G(P) = P \rtimes H$. Define
\[
b(G) = \frac{1}{|P'|} \cdot \frac{1}{[H : \mathbf{C}_H(P/P')]}.
\]
Then $\mathrm{D}_{p'}(G) \leq b(G)$.
\end{theorem}
\begin{proof}
Suppose for contradiction that the theorem is false, and let $G$ be a counterexample of minimal order, so that $b(G) < \mathrm{D}_{p'}(G)$. We have $|\mathrm{Irr}_{p'}(G)| = |\mathrm{Irr}_{p'}(\Norm{G}{P})|$ by the McKay Conjecture, so
\begin{align*}
   b(G) < \frac{|\mathrm{Irr}_{p'}(G)|}{|G|} = \frac{|\mathrm{Irr}_{p'}(\Norm{G}{P})|}{|G|} \le \frac{|\mathrm{Irr}_{p'}(\Norm{G}{P})|}{|\Norm{G}{P}|}=D_{p'}(\Norm{G}{P}). 
\end{align*}
If $|\Norm{G}{P}|<|G|$, then $\Norm{G}{P}$ satisfies the conclusion of the theorem by the minimality of $|G|$. But this is a contradiction since $b(G)=b(\Norm{G}{P})$ depends only on the local subgroup~$\Norm{G}{P}$. So we are left to consider the case that $\Norm{G}{P} = G$. Now, every $\chi \in \mathrm{Irr}_{p'}(G)$ restricted to~$P$ must be a sum of linear characters of $P$ by Clifford's Theorem, so $P' \le \ker \chi$. This means we may consider $\chi$ a character of $G/P'$, and so $|\mathrm{Irr}_{p'}(G)| \le |\Irr{G/P'}|$. Conversely, since $P/P'$ is an abelian normal subgroup of $G/P'$, we have for all $\chi \in \Irr{G/P'}$ that $\chi(1)$ divides $[G/P':P/P']=[G:P]$ by It{\^o}'s Theorem~\cite[Theorem~6.15]{I94}. So lifting $\chi$ to a character of $G$ yields an irreducible character of $p'$-degree. Hence, $|\mathrm{Irr}_{p'}(G)| = |\Irr{G/P'}|=k(G/P')$. 

Now we work in the quotient $G/P'=\q{G}=\q{P} \rtimes \q{H}$, which has a normal, abelian Sylow $p$-subgroup~$\q{P}$. Note that $\Phi(\q{P}) \nrm \q{G}$, being characteristic in $\q{P} \nrm \q{G}$. Moreover, $C_{\q{H}}(\q{P}) \nrm \q{G}$: it is normalized by $\q{H}$ since $\q{P} \nrm \q{G}$, and it is centralized, hence normalized, by $\q{P}$; so it is normal in $\q{G} = \q{P} \, \q{H}$. Recall that if $\q{N} \nrm \q{G}$, then $k(\q{G}) \le k(\q{N}) k(\q{G/N})$ by~\cite{Ga70}. Letting $\q{N} = \Phi(\q{P}) \Cen{\q{H}}{\q{P}}$ so that $|\q{N}|=|\Phi(\q{P})||\Cen{\q{H}}{\q{P}}|$, these being a $p$-group and a $p'$-group, and using the obvious inequality \mbox{$k(\q{N}) \le |\q{N}|$}, we then have
\begin{align*}
    k(\q{G}) \le |\Cen{\q{H}}{\q{P}}| \cdot |\Phi(\q{P})| \cdot k(\q{P}/\Phi(\q{P}) \rtimes \q{H}/\Cen{\q{H}}{\q{P}}).
\end{align*}
Now, the coprime action of $\q{H}/\Cen{\q{H}}{\q{P}}$ on the elementary abelian group $\q{P}/\Phi(\q{P})$ is faithful by a coprime action theorem~\cite[Corollary~3.30]{I08}; so $k(\q{P}/\Phi(\q{P}) \rtimes \q{H}/\Cen{\q{H}}{\q{P}}) \le |\q{P}/\Phi(\q{P})|$ by the $k(GV)$ Theorem~\cite{GMRS}. Additionally, $\q{H} \cong H$ since $P' \cap H = 1$. Altogether, we then finally have that
\begin{align*}
    |\mr{Irr}_{p'}(G)| = k(G/P') \le |\Cen{H}{P/P'}| \, |P/P'|.
\end{align*}
But then upon dividing through by $|G|=|P||H|$, we find that
\begin{align*}
    D_{p'}(G) \le \frac{|P/P'| |\Cen{H}{P/P'}|}{|P||H|}=b(G),
\end{align*}
a contradiction.
\end{proof}

As a corollary, we deduce a criterion for detecting abelian Sylow $p$-subgroups. This result was stated as Theorem~\ref{thmE} in the Introduction.

\begin{corollary}
Let $G$ be a finite group, $p$ a prime and $P$ a Sylow $p$-subgroup of $G$. If $D_{p'}(G) > 1/p$, then $P$ is abelian. If, in addition, $p$ is the smallest prime dividing $|G|$, then~$G$ has a normal $p$-complement.
\end{corollary}
\begin{proof}
By Theorem~\ref{t:Pabelian}, $1/p < D_{p'}(G) \le b(G)$; that is, $p > |P'| [H:\Cen{H}{P/P'}]$. Therefore, $|P'|=1$ and $[H:\Cen{H}{P}]<p$. In particular, $P$ is abelian. If, in addition, $p$ is the smallest prime dividing $|G|$, then $[H:\Cen{H}{P}]=1$. Therefore, $P \le \Zen{\Norm{G}{P}}$, and so~$G$ has a normal $p$-complement by Burnside's Transfer Theorem \cite[Theorem~5.13]{I08}.
\end{proof}

This criterion is best possible since an extraspecial $p$-group $P$ of order $p^{1+2m}$ has $p^{2m}$ representations of degree $1$ and $p-1$ irreducible representations of degree $p^m$. Therefore, $|\mathrm{Irr}_{p'}(P)|=p^{2m}$, so $D_{p'}(P)=1/p$ and $P$ is not abelian.

\, \\
\noindent \textbf{Acknowledgements.} The author thanks Attila Mar{\'o}ti, Hung Tong-Viet and the anonymous referee for their helpful comments on previous versions of the paper.

\, \\
\noindent \textbf{Statement on use of AI.} Claude Opus 5 (Anthropic) found Examples~\ref{ex:gen1} and~\ref{ex:X}, given the author's Example~\ref{ex:pi}, as well as Example~\ref{ex:refine} and the proof of Theorem~\ref{t:real}. All else was performed by the author, including the writing of the paper, and Claude was later used for proofreading and for independent verification of the computations. The author takes full responsibility for the contents of the paper.


\end{document}